\documentclass{article}[12 pt, english]

\usepackage{amssymb}

\newtheorem{Zae}{Zae}[section] 
\newtheorem{definition}[Zae]{Definition} 
\newtheorem{lemma}[Zae]{Lemma} 
\newtheorem{prop}[Zae]{Proposition}
\newtheorem{theorem}[Zae]{Theorem} 
 
\newtheorem{coro}[Zae]{Corollary} 
\newtheorem{example}[Zae]{Example}
\newtheorem{remark}[Zae]{Remark}
\newtheorem{conjecture}[Zae]{Conjecture}
\newcommand{\qed}{\raisebox{-.8ex}{$\Box$}}
\newenvironment{proof}
{\noindent{\bf Proof.}}
{\hfill \qed\\}

\begin{document}
\title{On the axioms defining a quadratic Jordan division algebras}      
\author{Matthias Gr\"uninger\footnote{supported by the ERC (grant \# 278 469)} \\
Universit\'e catholique de Louvain
 \\      
Institute pour la recherche en math\'ematique et physique 
\\ Chemin du Cyclotron 2, bte. L7.01.02 \\
1348 Louvain-la-Neuve, Belgique  \\
E-Mail:matthias.grueninger@uclouvain.be}    
\maketitle      
\begin{abstract} Quadratic Jordan algebras are defined by identities that have to hold strictly, i.e that
continue to hold in every scalar extension. In this paper we show that strictness is not required for quadratic 
Jordan division algebras.
\end{abstract}      

\section{Basics}
We begin with the classical definition of Jordan algebras.
\begin{definition}
Let $k$ be a field with $char k \ne 2$. 
\begin{enumerate}
\item
A commutative, unital $k$-algebra $J$ is called a (linear) Jordan algebra
if 
$${\rm (J)}\hspace{2cm} a^2 \cdot (ba) =(a^2 \cdot b) \cdot a$$
holds for all $a,b \in J$.
\item A non-zero element $a$ of Jordan algebra $J$ is called invertible if there is an element $a^{-1} \in J$ 
with $a \cdot a^{-1} =1$ and $a^2 \cdot a^{-1} =a$.
\item A Jordan algebra is called a Jordan division algebra if every non-zero element is invertible.
\end{enumerate}  
\end{definition}
The standard example for a Jordan algebra arises in the following way: Let $A$ be an associative $k$-algebra.
Define a new multiplication $\circ$ on $A$ by
$$a \circ b = \frac{1}{2} (ab +ba).$$
Then $A^+ =(J,\circ)$ is a Jordan algebra.\bigskip \\
As one can already see in this example, the constraint that $char k \ne 2$ is necessary.
One can see that a commutative algebra over a field of characteristic $2$ which satisfies (J) is associative.
But since Jordan algebras have been a useful tool to describe some algebraic groups that are also defined over
fields of even characteristic, one has to alter the definition of a Jordan algebra to include the case
characteristic $2$. In 1966 Kevin McCrimmon came up with new definition for Jordan algebras which works
for a field of arbitrary characteristic (see \cite{McCr}). For convenience, we first introduce some
more definitions.
\begin{definition}
Let $k$ be a field of arbitrary characteristic. 
\begin{enumerate}
\item If $V, W$ are two vector-spaces or $k$, then a map $Q: V\to W$ is called quadratic if
$Q(tv)=t^2 Q(v)$ for all $t \in k, v\in V$ and if there is a $k$-bilinear map $f: V\times V \to W$ 
with $Q(v+w)=Q(v) +Q(w) +f(v,w)$ for all $v,w \in V$. 
\item A quadratic algebra over $k$ is a pair $(J,Q)$ where $J$ is 
a $k$-vectorspace and $Q: J\to End_k(J): a \mapsto Q_a$ is quadratic.
\item 
For a quadratic algebra $(J,Q)$ and $a,b\in J$ one defines the maps $
Q_{a,b}, V_{a,b}: J \to J$ by $cQ_{a,b} =cQ_{a+b} -cQ_a -cQ_b$ and $cV_{a,b} =
bQ_{a,c}$ for $c\in J$. Of course, one has $cV_{a,b} =aV_{c,b}$ for all $a,b,c \in J$. The map
$J\times J \to End_k(J): (a,b) \mapsto V_{a,b}$ is $k$-bilinear.
\item An element $a \in J$ is called invertible if $Q_a$ is invertible. The element $a^{-1}:= aQ_a^{-1}$ 
is called the inverse of $a$. We denote the se t of invertible elements of $J$ by $J^*$. 
\item 
If $K$ is an extension field of $k$, then one defines the quadratic algebra $J_K:=(K \otimes_k J, \hat{Q})$ by
$\hat{Q}_{\sum_{i=1}^n t_i \otimes a_i} =\sum_{i=1}^n t_i^2 Q_{a_i} +\sum_{i <j} t_i t_j Q_{a_i,a_j}$. 
We say that an identity in $J$ holds \textbf{strictly} if it holds in $J_K$ for all extensions $K/k$.
\end{enumerate}
\end{definition}
\begin{definition}
Let $(J,Q)$ be a quadratic algebra over $k$ and $1 \in J^{\#}:= J\setminus \{0\}$. Then $(J,Q,1)$ is called a 
weak quadratic Jordan algebra if the following holds for all $a,b\in J$.
\begin{itemize}
\item[(QJ1)] $Q_1 =id_J$.
\item[(QJ2)] $Q_a V_{a,b} =V_{b,a} Q_a$.
\item[(QJ3)] $Q_{bQ_a} =Q_a Q_b Q_a$. 
\end{itemize}
A weak quadratic Jordan algebra is called a quadratic Jordan algebra if (QJ1)-(QJ3) hold strictly, i.e. if
$J_K$ is a weak quadratic Jordan algebra for all extension fields $K/k$.
\end{definition}
\begin{remark}\label{remark1}
\begin{enumerate}
\item If $J$ is a weak quadratic Jordan algebra and $a \in J$ is invertible, then we have
$Q_{a^{-1}} =Q_a^{-1}$. Indeed, we have $a =a^{-1} Q_a$ and thus with (QJ3)
$$Q_a =Q_{a^{-1}Q_a} = Q_a Q_{a^{-1}} Q_a,$$
hence $Q_a^{-1} =Q_{a^{-1}}$. Note that if $a,b \in J$ are invertible, then $aQ_b$ and $a^{-1}$ are also
invertible. 
\item (QJ1)-(QJ3) hold strictly iff their linearized versions holds. 
It is clear that (QJ1) always holds strictly if it holds, and
(QJ2) and (QJ3) hold strictly iff additionally
\begin{itemize}
\item[(QJ2*)]
 $Q_{a_1}V_{a_2,b} +Q_{a_1,a_2}V_{a_1,b}=V_{b,a_1}Q_{a_1,a_2}+V_{b,a_2}Q_{a_1}$. 
\item[(QJ3*)]
$Q_{bQ_{a_1},bQ_{a_1,a_2}}=Q_{a_1,a_2}Q_{b}Q_{a_1}+Q_{a_1}Q_{b}Q_{a_1,a_2}$.
\item[(QJ3**)] $Q_{bQ_{a_1,a_2}}+Q_{b Q_{a_2},bQ_{a_1}}=Q_{a_1}Q_{b}Q_{a_2}+Q_{a_2}Q_{b}Q_{a_1} 
+Q_{a_1,a_2}Q_{b}Q_{a_1,a_2}$.
\end{itemize}
hold for all $a_1,a_2,b \in J$. One sees easily that (QJ3) and (QJ3*) imply (QJ3**). If $|k| \geq 3$, then 
then (QJ2*) follows from (QJ2), and if $|k| \geq 4$, then (QJ3*) follows from (QJ3).   
 Thus if $|k| \geq 4$, then every weak quadratic
Jordan algebra is a quadratic Jordan algebra.
\item The author doesn't know an example of a weak quadratic Jordan algebra which is not a quadratic 
Jordan algebra.
\item Let $J$ be a linear Jordan algebra over $k$. For $a \in J$ define $Q_a: J \to J$ by
$bQ_a = - a^2 \cdot b +2a \cdot (a\cdot b)$. Then $(J,Q,1)$ is a quadratic Jordan algebra. If
$char k \ne 2$ and $J$ is a quadratic Jordan algebra, then we define a multiplication $\cdot $ on $J$
by $a \cdot b = \frac{1}{2} 1 Q_{a,b}$. Then one can show that $J$ is a linear Jordan algebra. Therefore 
for $char k \ne 2$ these two concepts coincide
(see for example \cite{McCr}). Moreover, an element is invertible in the linear Jordan 
algebra
iff it is invertible in the quadratic Jordan algebra.
\end{enumerate}
\end{remark}
\begin{example} 
Let $R$ be a unital, associative algebra over $k$. For $a \in R$ define $Q_a: R \to R: b\mapsto bab$. Then
$R^+:=(R,Q,1)$ is a quadratic Jordan algebra. 
\end{example}
We now introduce the important concept of an isotope.
\begin{definition}
Let $J=(J,Q)$ be a quadratic algebra and $a \in J^*$. We define the $a$-isotope $J^a =(J,Q^a)$ of $J$ by
$xQ^a_y = xQ_a^{-1} Q_y$ for all $x,y \in J$. 
\end{definition}
\begin{lemma}
Let $J$ be a quadratic algebra. For $a,b,c \in J$ we have
\begin{enumerate}
\item $Q^a_a = id_J$.
\item $Q^a_{b,c} = Q_a^{-1} Q_{b,c}$.
\item $V^a_{b,c} = V_{b,cQ_a^{-1}}$.
\end{enumerate}
\end{lemma}
\begin{proof}
(a) and (b) are clear. For (c) let $x \in J$. Then we have
$$xV^a_{b,c} =cQ^a_{b,x} = cQ_a^{-1}Q_{b,x} = xV_{b,cQ_a^{-1}}.$$
Thus the claim follows. 
\end{proof}
\begin{prop}
Let $J$ be a (weak) quadratic Jordan algebra and $a \in J^*$. Then $J^a=(J,Q^a,a)$ is also a (weak) quadratic 
Jordan algebra.
\end{prop}
\begin{proof}
It is clear that (QJ1) holds. Let $b,c,x \in J$. Then we have
$$xQ_b^a V^a_{b,c} = xQ_a^{-1} Q_b V_{b,cQ_a^{-1}} =xQ_a^{-1} V_{cQ_a^{-1},b} Q_b=
bQ_{cQ_a^{-1},xQ_a^{-1}} Q_b=$$
$$bQ_a^{-1} Q_{c,x} Q_a^{-1} Q_b= xV_{c bQ_a^{-1}} Q^a_b = xV^a_{c,b} Q^a_b.$$
This shows (QJ2). Moreover, we have
$$Q^a_{bQ^a_c} = Q_a^{-1} Q_{bQ_a^{-1} Q_c} = Q_a^{-1} Q_c Q_{bQ_a^{-1}} Q_c =Q_a^{-1} Q_c Q_a^{-1} Q_b
Q_a^{-1} Q_c =Q^a_c Q^a_b Q^a_c.$$
This shows (QJ3). Hence if $J$ is a weak quadratic Jordan algebra, so is $J^a$. Moreover, we have
$(K \otimes_k J)^a =K \otimes_k J^a$ for all extension fields $K/k$, thus $J^a$ is a quadratic Jordan algebra
if $J$ is.  
\end{proof}
\begin{definition}
\begin{enumerate}
\item If $(J,Q,1)$ and $(J^{\prime},Q^{\prime},1^{\prime})$ are weak quadratic Jordan algebras over $k$, 
then a Jordan homomorphism between $J$ and $J^{\prime}$ is a homomorphism 
$f: J \to J^{\prime}$ such that $f(1)=1^{\prime}$ and 
$f(aQ_b)=f(a)Q^{\prime}_{f(b)}$ for all $a,b \in J$ holds.
\item If $(J,Q,1)$ and $(J^{\prime},Q^{\prime},1^{\prime})$ are weak quadratic Jordan algebras over $k$, 
then $J$ and $J^{\prime}$ are called isotopic iff $J^{\prime}$ is isomorphic to an isotope of $J$. 
\item Let $J^{\prime}$ be a subspace of a weak quadratic Jordan algebra $J$. Then $J^{\prime}$ is called
a Jordan subalgebra of $J$ if $e \in J$ and if $J^{\prime} Q_a \subseteq J^{\prime}$ for all
$a \in J^{\prime}$ holds.
\item A quadratic Jordan algebra $J$ is called special if there is an associative $K$-algebra $R$ such 
that $J$ is isomorphic to a Jordan subalgebra of $R^+$. 
\end{enumerate}
\end{definition}
In this paper we are mainly interested in (weak) quadratic Jordan division algebras.
\begin{definition}
A (weak) quadratic Jordan algebra is called a (weak) quadratic Jordan division algebra if every 
non-zero element in $J$ is invertible.
\end{definition}
The theory of quadratic Jordan division algebras is connected with the theory of Moufang sets.
\begin{definition}
A Moufang set consists of a set $X$ with $|X| \geq 3$ and a family $(U_x)_{x\in X}$ of subgroups in $Sym X$ 
such that the following holds:
\begin{enumerate}
\item For all $x \in X$ the group $U_x$ fixes $x$ and acts regularly on $X \setminus \{x\}$.
\item For all $x,y \in X$ and all $g\in U_x$ we have $U_y^g =U_{yg}$.
\end{enumerate}
The groups $U_x$ are called the \textbf{root groups} of the Moufang set.
The group $G^{\dagger} :=\langle U_x; x\in X\rangle$ is called the little projective group of the Moufang 
set.
$G^{\dagger}$ is a $2$-transitive subgroup of $Sym X$. The Moufang set is called \textbf{proper} if $G^{\dagger}$ 
is not sharply $2$-transitive and \textbf{improper} else.
\end{definition}
\begin{example}
\begin{enumerate}
\item Let $X$ be a set with at least $3$ elements, $G \leq Sym X$ be a sharply $2$-transitive group. Then 
$(X,(G_x)_{x\in X})$ is an improper Moufang set with little projective group $G^{\dagger} =G$.  
\item Let $k$ be a field, $X:=\mathbb{P}^1(k)$ and $U_x$ be the subgroup of $PSL_2(k) \leq Sym X$ induced by the
group of unipotent matrices that fix $x$. Then $(X, (U_x)_{x\in X}$ is a Moufang set with little projective
group $G^{\dagger} = PSL_2(k)$. It is proper iff $|k| \geq 4$. 
\end{enumerate}
 \end{example}
 The second construction can be generalized to weak quadratic Jordan divsion algebras. In \cite{DW} the author 
 showed the following.
 \begin{theorem}\label{weakJordan}
 Every weak quadratic Jordan division algebra defines a Mou-\\fang set $\mathbb{M}(J)$ with
 root groups isomorphic to $(J,+)$. The algebra $J$ is determined by
 $\mathbb{M}(J)$ up to isotopy. 
\end{theorem}    
De Medts and Weiss didn't use the concept of a weak quadratic Jordan algebra and formulated their theorem for 
quadratic Jordan division algebras, but their proof doesn't make use of the strictness of (QJ1)-(QJ3), so 
it also holds for weak Jordan division algebras.\\
One of the big open problems concerning Moufang sets is the following conjecture:
\begin{conjecture}\label{RGC}
If $(X,(U_x)_{x\in X})$ is a proper Moufang set with $U_x$ abelian for all $x \in X$, then there is a field
$k$ and a quadratic Jordan division algebra $J$ over $k$ such that $(X,(U_x)_{x\in X})$ is isomorphic to
$\mathbb{M}(J)$.
\end{conjecture}
If \ref{RGC} is true, then one has a classification of proper Moufang sets with abelian root groups since
quadratic Jordan division algebras have been classified by McCrimmon and Zel'manov (see \cite{McZ}). 
The proof follows from the classification of simple quadratic Jordan algebras over an algebraically 
closed field, therefore it is essential that scalar extensions are allowed. 
There has been progress in proving \ref{RGC} (see \cite{DS}), but in general this conjecture is still open. 
However, if there would be a weak quadratic Jordan division algebra which is not a quadatric Jordan algebra, 
then conjecture \ref{RGC} would be false. Such an algebra could exist over $\mathbb{F}_2$ or $\mathbb{F}_3$. 
In this paper we will prove that no such algebra exists.
\bigskip \\
\textbf{MAIN THEOREM} Every weak quadratic Jordan division algebra 
is a quadratic Jordan algebra.
\section{Some useful identities}
In the following let $(J,Q,1)$ be a weak quadratic Jordan algebra.
\begin{lemma}\label{QJ2'}
$yQ_{aQ_x,x} = aQ_{yQ_x,x}$ for all $a,x,y \in J$.
\end{lemma}
\begin{proof}
(QJ2) implies
$xQ_{a,y} Q_x = y V_{a,x} Q_x =y Q_x V_{x,a}=aQ_{x,yQ_x} $. Since the first expression is symmetric in
$a$ and $y$, so is the second. Hence we get $aQ_{x,yQ_x} =yQ_{aQ_x,x}$.
\end{proof}
\begin{lemma}\label{identity} For all $x \in J$ we have
$Q_{x,1} =V_{x,1} =V_{1,x}$.
\end{lemma}
\begin{proof}
By (QJ2) we have
$V_{x,1} =V_{x,1}Q_1 = Q_1 V_{1,x} = V_{1,x}$. We have
$xQ_{1,y}= y V_{1,x} =yV_{x,1}=1 Q_{x,y}$ for all $y \in j$. Since the last expression is symmetric in $x$ 
and
$y$, so is the first. Thus we have $y V_{1,x} = xQ_{1,y} = yQ_{1,x} $.
\end{proof}
\begin{lemma}
\label{isotope}
If $a \in J^*$ and $b \in J$, then we have
$V_{x,a^{-1}}=V_{a,xQ_a^{-1}} =Q_a^{-1} Q_{x,a}$.
\end{lemma}
\begin{proof}
We apply \ref{identity} for the isotope $J^a$ and have
$$V^a_{x,a}=V^a_{a,x}=Q^a_{a,x}$$
and therefore
$$V_{x,a^{-1}}=V_{a,xQ_a^{-1}}=Q_a^{-1}Q_{a,x}.$$
\end{proof}
\begin{lemma}\label{2.4}
$Q_{1,x} Q_x =Q_x Q_{1,x}$ for all $x \in J$.
\end{lemma}
\begin{proof}
We have $Q_x Q_{1,x} =Q_x V_{x,1} = V_{1,x} Q_x = Q_{1,x} Q_x$ by (QJ2) and \ref{identity}.
\end{proof}
\begin{lemma}\label{2.5}
If $x \in J^*$, then $Q_x^{-1} V_{a,x} = V_{x,a}Q_x^{-1} =Q_{a,x^{-1}}$ for all $a \in J$.
\end{lemma}
\begin{proof}
We have
$aQ_{yQ_x,x}=yQ_{aQ_x,x} = yQ_x Q_{a,x^{-1}} Q_x$ 
for all $a \in J$ by (QJ3) and \ref{QJ2'}. Replacing $y$ by $yQ_x^{-1}$, we get
$yV_{x,a} = aQ_{y,x} = yQ_{a,x^{-1}} Q_x$.
Thus the second equation follows. The first now follows from (QJ2).
\end{proof}
\begin{lemma}\label{2.6} If $x,y \in J^*$, then we have
$$Q_x^{-1} Q_{x+y} Q_y^{-1} = Q_{x^{-1} +y^{-1}}.$$
\end{lemma}
\begin{proof}
Using the previous lemma for $x^{-1}$ and $y$, we have
$$Q_x^{-1} Q_{x,y} Q_y^{-1} = V_{y,x^{-1}}Q_y^{-1}=Q_{y^{-1},x^{-1}}.$$
Since $Q_x^{-1}Q_x Q_y^{-1} = Q_y^{-1}$ and $Q_x^{-1} Q_y Q_y^{-1} =Q_x^{-1}$, we get
$$Q_x^{-1} Q_{x+y} Q_y^{-1} = Q_x^{-1}(Q_x +Q_y +Q_{x,y})Q_y^{-1}=Q_y^{-1}+Q_x^{-1}+Q_{x^{-1},y^{-1}}=
Q_{x^{-1}+y^{-1}}.$$
\end{proof}
\bigskip \\
We will also make use of the following "Hua-identity" for weak quadratic Jordan division algebras. It was
proved by De Medts and Weiss in \cite{DW} in order to show that a quadratic Jordan division algebra defines
a Moufang set. As mentioned before, the proof doesn't make use of the strictness of (QJ1)-(QJ3), so it still
holds for weak quadratic Jordan division algebras.
\begin{theorem}\label{Hua}
Let $J$ be a weak quadratic Jordan divsion algebra and $a,b \in J^*$ with $a \ne b^{-1}$. Then we have
$$aQ_b =b-(b^{-1}-(b-a^{-1})^{-1})^{-1}.$$
\end{theorem}
\section{Derivations and anti-derivations of weak quadratic Jordan algebras}
\begin{definition}
Let $J$ be a weak quadratic Jordan algebra and $\epsilon \in \{+,-\}$. 
A linear map $\delta: J\to J$ is called an $\epsilon$-derivation if $\delta(aQ_b) =
\epsilon \delta(a)Q_b +aQ_{b,\delta(b)}$ holds
for all $a,b \in J$.
\end{definition}
We will call the $+$-derivations just \textit{derivations} and the $-$-derivations \textit{anti-derivations}. 
\begin{example} Let $A$ be an associative algebra and $J \subseteq A$ a special quadratic Jordan algebra.
If $\delta:A \to A$ is a(n anti-)derivation of $A$ with $\delta(J) \leq J$, then $\delta$ induces a(n anti-)
derivation of $J$. Indeed, for $a,b \in J$ we have $\delta(aQ_b) =\delta(bab) = \delta(b)ab +
\epsilon b\delta(ab) =
\delta(b)ab +\epsilon b\delta(a) b +\epsilon^2 ba\delta(b) = \epsilon \delta(a)Q_b+aQ_{b,\delta(b)}$ 
with $\epsilon=+$ if $\delta$ is a derivation and $\epsilon=-$ for $\delta$ an anti-derivation.
\end{example}
\begin{lemma}\label{lemma}
Let $\delta$ be an $\epsilon$-derivation for $\epsilon =\pm $ and $a,b,c \in J$. Then we have:
\begin{enumerate}
\item $\delta(aQ_{b,c}) =
\epsilon \delta(a)Q_{b,c}+aQ_{\delta(b),c} +aQ_{b,\delta(c)}$ 
and
$\delta(aV_{b,c}) = \delta(a)V_{b,c} + aV_{\delta(b),c} + \epsilon aV_{b,\delta(c)}$ for all $a,b,c \in J$.
\item If $a \in J^*$, then $\delta(a^{-1})= -\epsilon \delta(a)Q_a^{-1}$.
\item The identity is an anti-derivation.
\item If $char k \ne 2$ and $\delta$ a derivation, then $\delta(1) =0$.
\item If $char k =2$ and $\delta$ is a derivation, then $Q_{1,\delta(1)}=0$.
\item If $char k \ne 2$ and $\delta$ is an anti-derivation, then
$\delta(a) =\frac{1}{2} aQ_{1,\delta(1)}$.
\end{enumerate}
\end{lemma}
\begin{proof}
\begin{enumerate}
\item The first equation follows by linearizing the defining property of an $\epsilon$-derivation. The second 
equation 
can be obtained by the first.
\item We have $\delta(a^{-1}) =\delta(aQ_{a^{-1}}) = \epsilon \delta(a)Q_{a^{-1}} +aQ_{a^{-1},\delta(a^{-1})}$.
Now $aQ_{a^{-1},\delta(a^{-1})} =aV_{a,\delta(a^{-1})} Q_a^{-1}=\delta(a^{-1})Q_{a,a}Q_a^{-1} =
2\delta(a^{-1})$ by \ref{2.5}. Thus we get $-\delta(a^{-1})= \epsilon\delta(a)Q_{a^{-1}} =\epsilon\delta(a)Q_a^{-1}$.
\item We have $id(aQ_b) =aQ_b=-id(a)Q_b +a Q_{b,id(b)}$, which shows that the identity is an anti-derivation.
\item We have
$\delta(1) =\delta(1^{-1}) =-\delta(1)Q_1^{-1} =-\delta(1)$, thus the claim follows.
\item For all $a \in J$ we have
$$\delta(a) =\delta(aQ_1) =\delta(a)Q_1 +aQ_{1,\delta(1)}=\delta(a)+aQ_{1,\delta(1)},$$ hence the
claim follows. 
\item We have
$$\delta(a) =\delta(aQ_1) =-\delta(a)Q_1 +aQ_{1,\delta(1)} =-\delta(a)+aQ_{1,\delta(1)}$$
and thus
$\delta(a) =\frac{1}{2}aQ_{1,\delta(1)}$. 
\end{enumerate}
\end{proof}
We set $\mathfrak{D}_{\epsilon}(J):=\{\delta \in End_k(J); \delta$ is an $\epsilon$-derivation of $J\}$ and
$\mathfrak{D}(J) =\mathfrak{D}_1(J) +\mathfrak{D}_{-1}(J)$. If $char k =2$, then we have
$\mathfrak{D}_+(J) =\mathfrak{D}_{-}(J) =\mathfrak{D}(J)$, while for $char k\ne 2$ we have
$\mathfrak{D}(J) =\mathfrak{D}_+(J) \oplus \mathfrak{D}_{-}(J)$. We call the elements of
$\mathfrak{D}(J)$ {\it generalized derivations}.
\begin{lemma}
For $\epsilon_1,\epsilon_2 \in \{+,-\}$ we have
$[\mathfrak{D}_{\epsilon_1}(J),\mathfrak{D}_{\epsilon_2}(J)] \subseteq \mathfrak{D}_{\epsilon_1 \epsilon_2}(J)$.
Especially $\mathfrak{D}_+(J)$ and $\mathfrak{D}(J)$ are Lie subalgebras of $End_k(J)$, and if $char k \ne 
2$, then $\mathfrak{D}(J)$ is $Z_2$-graded.
\end{lemma}
\begin{proof}
For $i=1,2$ let $\delta_i \in \mathfrak{D}_{\epsilon_i}(J)$. 
For $a,b\in J$ we have
$$\delta_1 (\delta_2(aQ_b)) =\delta_1( \epsilon_2 \delta_2(a)Q_b +aQ_{b,\delta_2(b)})=$$
$$ \epsilon_1 \epsilon_2 \delta_1(\delta_2(a))Q_b + \epsilon_2 \delta_2(a)Q_{b,\delta_1(b)}+
\epsilon_1 \delta_1(a)Q_{b,\delta_2(b)} +aQ_{\delta_1(b),
\delta_2(b)}+aQ_{\delta_1(\delta_2(b)),b}$$
and analogously 
 $$\delta_2 (\delta_1(aQ_b))=
 \epsilon_1 \epsilon_2 \delta_2(\delta_1(a))Q_b +
 \epsilon_1 \delta_1(a)Q_{b,\delta_2(b)} +\epsilon_2 \delta_2(a)Q_{b,\delta_1(b)}$$
 $$+aQ_{\delta_1(b),\delta_2(b)}+aQ_{\delta_2(\delta_1(b)),b}.$$
 Thus we get $$[\delta_1,\delta_2](aQ_b) = \epsilon_1 \epsilon_2 
 \delta_1(\delta_2(a))Q_b-\delta_2(\delta_1(a))Q_b+
 aQ_{\delta_1(\delta_2(b)),b}-aQ_{\delta_1(\delta_2(b)),b}=$$
 $$\epsilon_1 \epsilon_2 [\delta_1,\delta_2](a)Q_b +aQ_{[\delta_1,\delta_2](b),b}.$$
 Hence the claim follows.
\end{proof}
\begin{lemma}\label{Q_1,a} 
If $J$ is a quadratic Jordan algebra, then for all $a \in J$ the map $Q_{1,a}$ is an 
anti-derivation of $J$.
\end{lemma}
\begin{proof}
$Q_{1,a}$ is an anti-derivation iff for all $b \in J$ we have
$$Q_{b}Q_{1,a}=-Q_{1,a} Q_b +Q_{b,bQ_{1,a}}.$$
But this is just identity (QJ3*) with $a_1=1$ and $a_2=a$ which holds by definition in quadratic Jordan 
algebras.
\end{proof}
\begin{remark} Let $J$ be a linear Jordan algebra. 
For $a,b,x \in J$ define the associator 
$\{a,x,b\}:=(a\cdot x)\cdot b-a \cdot (x \cdot b)$. Now 
by \ref{remark1}(d) $Q_{1,a}$ corresponds to the map 
$x \mapsto 2 a\cdot x$. Thus for all $a,b \in J$ the map $[Q_{1,a},Q_{1,b}]$ is a derivation of $J$.
Hence $$x[Q_{1,a},Q_{1,b}]=4 (b \cdot (a \cdot x) -a \cdot (b \cdot x)) =4 ((a \cdot x) \cdot b -
(a \cdot (x \cdot b) ))=4\{a,x,b\}.$$
Thus the map $x \mapsto \{a,x,b\}$ is a derivation of $J$.
Moreover, if $\delta$ is an anti-derivation of (the quadratic Jordan algebra) $J$, then
$\delta(a)=2 a \cdot \delta(1)$ by \ref{lemma}(f). One can easily prove that a linear
map $\delta: J \to J$ is a derivation of (the quadratic Jordan algebra) $J$ iff
$\delta(a\cdot b)=\delta(a)\cdot b + a\cdot \delta(b)$ for all $a,b \in J$, which is the usual definition of 
a derivation of a linear Jordan algebra.
\end{remark}
\begin{theorem}
Let $J$ be a weak quadratic Jordan algebra. Suppose that for all $a,y \in J^*$ there is a
generalized derivation $\delta$ with $\delta(a) =y$. 
Then $J$ is a quadratic Jordan algebra. 
\end{theorem}
\begin{proof}
Let $a \in J$. We set
$$L_J(a):=\{y \in J: V_{b,y}Q_a +V_{b,a}Q_{a,y}=Q_a V_{y,b} +Q_{a,y} V_{a,b} {\rm \ and}$$
$$Q_{bQ_a,b_{a,y}} =Q_aQ_b Q_{a,y}+Q_{a,y}Q_b Q_a {\rm \ for \ all} \ b\in J\}.$$
Then $L_J(a)$ is a subspace of $J$. 
Now let $b,x \in J$. Then we have
$$\delta(xQ_a Q_b Q_a) =\epsilon \delta(x Q_a Q_b) Q_a +xQ_a Q_b Q_{a,\delta(a)} = $$ 
$$\delta(xQ_a)Q_b Q_a+
\epsilon xQ_a Q_{b,\delta(b)} Q_a +xQ_a Q_b Q_{a,\delta(a)} =$$
$$\epsilon \delta(x) Q_a Q_b Q_a+x Q_{a,\delta(a)} Q_b Q_a +\epsilon xQ_a Q_{b,\delta(b)} Q_a +xQ_a Q_b Q_{a,\delta(a)}.$$
On the other side,
$$\delta(x Q_a Q_b Q_a) =
 \delta(xQ_{bQ_a}) = \epsilon \delta(x)Q_{bQ_a} +xQ_{bQ_a,\delta(bQ_a)} =$$
 $$\epsilon \delta(x)Q_a Q_b Q_a +
xQ_{bQ_a,\epsilon \delta(b)Q_a}+xQ_{bQ_a,bQ_{a,\delta(a)}}=$$
$$\epsilon \delta(x) Q_b Q_a Q_b +\epsilon xQ_a Q_{b,\delta(b)} Q_a+xQ_{bQ_a,bQ_{a,\delta(a)}}.$$
Hence we get
$$Q_{a,\delta(a)} Q_b Q_a +Q_a Q_b Q_{a,\delta(a)} =Q_{bQ_a,bQ_{a,\delta(a)}}.$$
Moreover, we have
$$\delta(xV_{b,a} Q_a ) = \epsilon \delta(xV_{b,a})Q_a+xV_{b,a} Q_{a,\delta(a)} =$$
$$\epsilon \delta(x)V_{b,a}Q_a+xV_{b,\delta(a)}Q_a+\epsilon xV_{\delta(b),a}Q_a+xV_{b,a} Q_{a,\delta(a)}.$$
On the other side, we have
$$\delta(xV_{b,a}Q_a) =\delta(xQ_a,V_{a,b}) = \delta(xQ_a)V_{a,b} +xQ_aV_{\delta(a),b}+\epsilon 
xQ_aV_{a,\delta(b)}=$$
$$\epsilon \delta(x)Q_a V_{a,b} +xQ_{a,\delta(a)}V_{a,b}+xQ_aV_{\delta(a),b}+\epsilon xQ_a V_{a,\delta(b)}.$$
Thus we get
$$V_{b,\delta(a)}Q_a +V_{b,a} Q_{a,\delta(a)} = Q_a V_{\delta(a),b}+Q_{a,\delta(a)} V_{a,b}.$$
This shows that $\delta(a) \in L_J(a)$ for all $\delta \in \mathfrak{D}_{\epsilon}(J)$, $\epsilon=\pm$. 
Thus the claim follows.
\end{proof}
\section{The proof of the main theorem}
\begin{theorem}
Let $J$ be a weak Jordan division algebra, $\epsilon \in \{+,-\}$ and $\delta \in End_k(J)$ with $\delta(a^{-1}) =
-\epsilon \delta(a)Q_a^{-1}$ 
for all $a \in J^*$. Then $\delta$ is an $\epsilon$-derivation.
\end{theorem}
\begin{proof}
Let $a,b \in J^*$ with $a \ne b^{-1}$. Then we have by the Hua-identity
$$aQ_b = b-(b^{-1}-(b-a^{-1})^{-1})^{-1}.$$
Thus we get
$$\delta(aQ_b) =\delta(b) -\delta(b^{-1}-(b-a^{-1})^{-1})^{-1})=$$
$$\delta(b) +\epsilon (\delta(b^{-1})-\delta((b-a^{-1})^{-1}))Q_{b^{-1}-(b-a^{-1})^{-1}}^{-1}=$$
$$\delta(b) +\epsilon (-\epsilon \delta(b) Q_b^{-1}+\epsilon \delta(b-a^{-1})Q_{b-a^{-1}}^{-1})
Q_{b^{-1}-(b-a^{-1})^{-1}}^{-1}=$$
$$\delta(b) -\delta(b) Q_b^{-1}Q_{b^{-1}-(b-a^{-1})^{-1}}^{-1} +\delta(b) Q_{b-a^{-1}}^{-1}
Q_{b^{-1}-(b-a^{-1})^{-1}}^{-1}+$$ 
$$\epsilon \delta(a)Q_a^{-1} Q_{b-a^{-1}}^{-1} Q_{b^{-1}-(b-a^{-1})^{-1}}^{-1}.$$
Now for $x= b^{-1}$ and $y = -(b-a^{-1})^{-1}$ we have by \ref{2.6}
$$Q_{b^{-1}-(b-a^{-1})^{-1}} = Q_{b^{-1}} Q_{b-(b-a^{-1})} Q_{-(b-a^{-1})^{-1}}=$$
$$Q_b^{-1} Q_a^{-1} Q_{b-a^{-1}}$$
and with $x = -(b-a^{-1})^{-1}$ and $y=b^{-1}$ we get
$$Q_{b^{-1}-(b-a^{-1})^{-1}}=Q_{-(b-a^{-1})^{-1}+b^{-1}} =Q_{b-a^{-1}}^{-1} Q_a^{-1} Q_b^{-1}.$$
Thus we get using \ref{2.5} and (QJ2)
$$\delta(aQ_b) =\delta(b) -\delta(b) Q_a Q_{b-a^{-1}} +\delta(b)Q_a Q_b +\epsilon \delta(a)Q_b=$$
$$\epsilon \delta(a)Q_b +\delta(b)Q_a (Q_{-a^{-1}}-Q_{b-a^{-1}}+Q_b) =\epsilon\delta(a)Q_b -
\delta(b)Q_a Q_{b,-a^{-1}}=$$
$$\epsilon \delta(a)Q_b +\delta(b)Q_a Q_{a^{-1},b} =\epsilon \delta(a) Q_b +\delta(b)V_{b,a} =$$
$$\epsilon \delta(a)Q_b +aQ_{b,\delta(b)}.$$
as desired. \\
We still have to prove $\delta(aQ_b) = \epsilon \delta(a)+ aQ_{b,\delta(b)}$ for $b \in \{0,a^{-1}\}$. 
The statement is clear for $b=0$, while we have by \ref{2.5}
$$\epsilon \delta(a)Q_{a^{-1}}+aQ_{a^{-1},\delta(a^{-1})}= \epsilon \delta(a)Q_a^{-1} +
a V_{a,\delta(a^{-1})} Q_a^{-1}=$$
$$-\epsilon^2 \delta(a^{-1})+\delta(a^{-1})Q_{a,a}Q_a^{-1}=- \delta(a^{-1})+2\delta(a^{-1})=
\delta(a^{-1})=\delta(aQ_{a^{-1}}).$$

\end{proof}

\begin{lemma}\label{4.2} Let $J$ be a weak quadratic Jordan division algebra.
Then for all $a \in J$ the map $\delta_a=Q_{1,a}$ is an anti-derivation of $J$.
\end{lemma}
\begin{proof}
We have by \ref{identity}, \ref{2.4} and \ref{2.5}
$$\delta_a(x^{-1}) =x^{-1} Q_{1,a} = x^{-1} V_{1,a} =aQ_{1,x^{-1}} = aQ_x^{-1} V_{1,x} =$$
$$aV_{1,x} Q_x^{-1} = xQ_{1,a} Q_x^{-1} = \delta_a(x)Q_x^{-1}$$ 
for all $x \in J^*$.
\end{proof}
\begin{coro}\label{4.3} For all $a \in J^*$ and all $b\in J$ the map $Q_{a,b}^a =V_{a,b}^a =V_{b,a}^a$ is an anti-derivation 
of $J^a$.
\end{coro}
For odd characteristic we can show that the converse of \ref{Q_1,a} holds. Thus \ref{4.2} implies that
a weak quadratic Jordan division algebra in odd characteristic  is a quadratic Jordan algebra. 
\begin{theorem}\label{linearjordanalgebra}
Let $J$ be a weak quadratic Jordan algebra over a field $k$ with $char k \ne 2$. 
Suppose that for all $a \in J$ the map $Q_{1,a}$ is an anti-derivation of $J$.
For $a,b \in J$ define $a \cdot b = \frac{1}{2} aQ_{1,b}$. Then $(J,+,\cdot)$ is 
a linear Jordan division algebra. Thus $J$ is a quadratic Jordan algebra.
\end{theorem}
\begin{proof}
We have $a \cdot b = \frac{1}{2} aQ_{1,b} =\frac{1}{2} aV_{1,b} = \frac{1}{2}bQ_{1,a} = b\cdot a$ 
by \ref{identity}, so
$\cdot$ is commutative. Moreover, we have $1\cdot a = a\cdot 1 =\frac{1}{2}aQ_{1,1} = \frac{1}{2} 2aQ_1 =a$, 
so
$1$ is the neutral element.
It remains to show that $a^2 \cdot (b\cdot a) = (a^2 \cdot b) \cdot a$ holds for all $a,b \in J$. Note that
$a^2 = \frac{1}{2}aQ_{1,a} = \frac{1}{2}1Q_{a,a} = 1Q_a$. Since $Q_{1,a\cdot b}$ 
is an anti-derivation, we have
$$a^2 \cdot (a\cdot b) = \frac{1}{2} 1Q_a Q_{1,a\cdot b} =-\frac{1}{2}1Q_{1,a\cdot b} Q_a +
\frac{1}{2} 1Q_{a,aQ_{1,a\cdot b}}=$$
$$-(a\cdot b)Q_a +\frac{1}{2}aQ_{1,2 a\cdot (a\cdot b)}=-(a\cdot b)Q_a +2 a \cdot (a \cdot (a\cdot b)).$$
Moreover, we have
$$(a^2 \cdot b) \cdot a = \frac{1}{4} 1Q_a Q_{1,b} Q_{1,a} =-\frac{1}{4}1Q_{1,b} Q_a Q_{1,a} +
\frac{1}{4} 1 Q_{a,aQ_{1,b}} Q_{1,a} = $$
$$-\frac{1}{2} b Q_{1,a} Q_a +\frac{1}{4} aQ_{1,2a \cdot b} Q_{1,a} = -(a\cdot b)Q_a+(a \cdot (a\cdot 
b))Q_{1,a}=$$
$$-(a\cdot b)Q_a +2(a \cdot (a\cdot b)) \cdot a = -(ab)Q_a +2 a\cdot (a\cdot (a\cdot b)).$$
Thus $(J,\cdot)$ is a linear Jordan algebra. 
\end{proof}
\bigskip \\
The following proof works for a field in arbitrary characteristic.
\begin{theorem}
A weak quadratic Jordan division algebra is a Jordan division algebra.
\end{theorem}
\begin{proof}
We have to show (QJ2*) and (QJ3*), the first only for $k=\mathbb{F}_2$. 
Since these equalities automatically hold if one of the 
elements involved is zero, we only have to show them for non-zero elements.
So let $a,b,c \in J^*$. Since $Q_{a,b}^c$ is an anti-derivation of $J^c$ by \ref{4.3}, we have
$$Q^c_a Q^c_{b,c}=-Q^c_{b,c}Q^c_a + Q^c_{aQ^c_{b,c},a},$$
hence
$$Q_c^{-1} Q_a Q_c^{-1} Q_{b,c} =-Q_c^{-1} Q_{b,c} Q_c^{-1} Q_a + Q_c^{-1} Q_{aQ_c^{-1} Q_{b,c},a}.$$
Multiplying $Q_c$ on the left yields
$$Q_a Q_c^{-1} Q_{b,c} +Q_{b,c} Q_c^{-1} Q_a + Q_{aQ_c^{-1} Q_{b,c},a}.$$
Replacing $a$ by $aQ_c$ and applying (QJ3) yields
$$Q_c Q_a Q_{b,c} +Q_{b,c} Q_a Q_c = Q_{aQ_{b,c},aQ_c}.$$
This shows (QJ3*).\\
We now show (QJ2*). Since $Q_{b,c}^c =V_{b,c}^c =V_{c,b^c}$ is an anti-derivation of $J^c$, we have
$$Q_a^c V_{b,c}^c =- V_{b,c}^c Q_a +Q^c_{aV_{b,c}^c,a}$$ 
and thus
$$Q_c^{-1} Q_a V_{b,c^{-1}} = -V_{b,c^{-1}} Q_c^{-1} Q_a +Q_c^{-1} Q_{aV_{b,c^{-1}},a}.$$
Replacing $c$ by $c^{-1}$ and applying (QJ2) yields
$$Q_c Q_a V_{b,c} = -Q_c V_{c,b} Q_a +Q_c Q_{cQ_{b,a},a}.$$
Multiplying $Q_c^{-1}$ on the left yields
$$(*) \hspace{2cm} Q_a V_{b,c} =-V_{c,b} Q_a+Q_{cQ_{b,a},a}.$$
Moreover, since $Q_{a,c}^a =V_{a,c}^a=V_{c,a}^a$ is an anti-derivation of $J^a$, we have
$$Q_{a,b}^a V^a_{a,c} = -V^a_{a,c} Q_{a,b}^a +Q^a_{aV_{a,c}^a,b} +Q^a_{a,bV^a_{a,c}}=
-V^a_{c,a} Q_{a,b}^a +Q^a_{aV_{a,c}^a,b} +Q^a_{a,bV^a_{a,c}}.$$
Hence we have
$$Q_a^{-1}Q_{a,b} V_{a,cQ_a^{-1}} =-V_{c,a^{-1}} Q_a^{-1} Q_{a,b} +Q_a^{-1} 
Q_{2c,b}+Q_a^{-1}Q_{a,bV_{a,cQ_a^{-1}}}.$$
Using (QJ2) and multiplying $Q_a$ on the left yields
$$Q_{a,b} V_{a,cQ_a^{-1}} =-V_{a^{-1},c} Q_{a,b} +2Q_{c,b} +Q_{a,cQ_a^{-1} Q_{a,b}}.$$
Replacing $c$ by $cQ_a$ yields
$$Q_{a,b} V_{a,c} =-V_{a^{-1},cQ_a} Q_{a,b} +2Q_{cQ_a,b} +Q_{a,cQ_{a,b}}.$$
By \ref{isotope} we have 
$V_{a^{-1},cQ_a} =V_{c,a}$. Hence we get
$$(\dagger) \hspace{2cm} Q_{a,b} V_{a,c} =-V_{c,a} Q_{a,b} +2Q_{cQ_a,b} +Q_{a,cQ_{a,b}}.$$
Adding $(*)$ and $(\dagger)$ yields
$$Q_{a,b} V_{a,c} + Q_a V_{b,c} = -V_{c,b} Q_a -V_{c,a} Q_{a,b} +2Q_{cQ_a,b} +2Q_{a,cQ_{a,b}}.$$
This gives (QJ2*) for $char k=2$.   
\end{proof}  
\section{Application for Moufang sets}
Let $\mathbb{M}=(X, (U_x)_{x\in X})$ be a proper Moufang set with abelian root groups.
$\mathbb{M}$ can be written in the form $M(U,\tau)$ with $U$ an (additively written) group isomorphic 
to a root group of $\mathbb{M}$ and $\tau$ a permutation of $U \cup \{\infty\}$ interchanging 
$0$ and $\infty$, where $\infty$ is a symbol not contained in $U$. \\
 By \cite{S} $\mathbb{M}$ is special, so by \cite{Tim}, Thm. 5.2(a) $U$ is 
either torsion free and uniquely divisible or an elementary-abelian $p$-group for a prime $p$. 
We write $char U =0$ in the first case and $char U =p$ in the second. We can view $U$ as a $k$-vectorspace 
for $k =\mathbb{Q}$ if $char U =0$ and $k =\mathbb{F}_p$ if $char U =p$. \\
In order to give $U$ the structure of a quadratic Jordan algebra, we need a 
quadratic map between $U$ and $End_k(U)$. There is a natural candidate for this map. Choose $e \in U^{\#}
=U \setminus \{0\}$ and set $h_a:=\mu_e \mu_a$ for $a \in U^{\#}$ and $h_0 =0$. Let $\mathcal{H}: U \to
End_k(U); a \mapsto h_a$ 
(see \cite{DS} for the definition of $\mu_a$). Then $(U,\mathcal{H},e)$ satisfies (QJ1) and (QJ3) and one has
$h_{a\tau} =h_a^{-1}$ and $h_{a \cdot s} = h_a \cdot s^2$ for $a \in U$ and all $s \in k$. Moreover, if $(U,\mathcal{H},e)$ is a quadratic Jordan divsion algebra, then $\mathbb{M} \cong \mathbb{M}(U,\tau)$. 
It remains to show that (QJ2) holds and 
that the map $(a,b) \mapsto h_{a,b}= h_{a+b}-h_a -h_b$ is biadditive. In \cite{DS} the authors showed 
the following:
\begin{theorem}\label{QJ2}
If $char U \ne 2,3$ and if (QJ2) holds, then $(U,\mathcal{H},e)$ is a quadratic Jordan division algebra.
\end{theorem}   
The authors had to exclude the case $char U \in \{2,3\}$ because (QJ1)-(QJ3) are required to hold strictly, 
which was only guaranteed if $|k| \geq 4$. But our main theorem shows that this is always the case 
for weak Jordan division algebras. Thus we get
\begin{coro}
\ref{QJ2} also holds for $char U \in \{2,3\}$.
\end{coro}
\begin{remark}
\begin{enumerate}
\item It is sufficient to prove a weaker version of axiom (QJ2) which has to hold in all isotopes of 
$(U,\mathcal{H},e)$, i.e. for all choices of $e \in U \setminus \{0\}$, compare 5.6 of \cite{DS}.
\item If $char U \ne 2,3$, then in order prove that $(U,\mathcal{H},e)$ is a quadratic Jordan division 
algebra, it is also sufficient to prove that the map $(a,b) \mapsto h_{a,b}$ is biadditive (5.12 of
\cite{DS}). In this case however the strictness is not the only obstacle for $char U \in \{2,3\}$ and
therefore it is not yet clear if the statement is also true in this case.
\end{enumerate}
\end{remark}

\end{document}